\documentclass{scrartcl}
\usepackage[T1]{fontenc}
\usepackage[utf8]{inputenc}
\usepackage{amsmath, amsthm, amssymb}
\usepackage{dsfont}
\usepackage{graphicx}
\usepackage{color}
\usepackage[format=plain]{caption}
\usepackage{capt-of}
\usepackage{subcaption}
\usepackage{multicol}
\usepackage{bm}
\usepackage{dlfltxbcodetips} 

\newcommand{\norm}[1]{\left\| #1 \right\|}  
\newcommand{\N}{\mathbb{N}}  

\newcommand{\R}{\mathbb{R}}

\newcommand{\spn}{\operatorname{span}}

\newcommand{\eps}{\varepsilon}
\renewcommand{\phi}{\varphi}
\newcommand{\ul}{\underline}
\newcommand{\ol}{\overline}

\newcommand{\diam}{\operatorname{diam}}

\numberwithin{equation}{section}

\newtheorem{thm}{Theorem}[section]
\newtheorem{cor}[thm]{Corollary}
\newtheorem{prop}[thm]{Proposition}
\newtheorem{lm}[thm]{Lemma}
\newtheorem{cond}[thm]{Condition}

\theoremstyle{definition} \newtheorem{algo}[thm]{Algorithm}
\theoremstyle{definition} \newtheorem{ex}[thm]{Example}
\theoremstyle{definition}

\title{Approximate solutions of convex semi-infinite optimization problems in finitely many iterations}

\author{Jochen Schmid and Miltiadis Poursanidis\\  
\small Fraunhofer Institute for Industrial Mathematics (ITWM), 67663 Kaiserslautern, Germany\\ 
\small jochen.schmid@itwm.fraunhofer.de}  

\date{}

\begin{document}

\maketitle

\begin{abstract}
\small{ \noindent 
We develop two 
adaptive discretization algorithms for convex semi-infinite optimization, 
which terminate after finitely many iterations at approximate solutions of arbitrary precision. In particular, they terminate at a feasible point of the considered optimization problem. Compared to the existing finitely feasible algorithms for general semi-infinite optimization problems, our algorithms work with considerably smaller discretizations and are thus computationally favorable. 
Also, our algorithms terminate at approximate solutions of arbitrary precision, while for general semi-infinite optimization problems the best possible approximate-solution precision can be arbitrarily bad. All occurring finite optimization subproblems in our algorithms have to be solved only approximately, and continuity is the only regularity assumption on our objective and constraint functions. Applications to parametric and non-parametric regression problems under shape constraints are discussed. 
}
\end{abstract}


\section{Introduction}

In this paper, we are concerned with convex semi-infinite optimization~\cite{HeKo93, Polak, Reemtsen, Stein, St12} 
that is, optimzation problems of the form 
\begin{align} \label{eq:SIP}
\min_{x\in X} f(x) \quad \text{s.t.} \quad g_i(x,y) \le 0 \quad \text{for all } y \in Y \text{ and } i \in I
\end{align}
with a (strictly) convex objective function $f: X \to \R$ on some compact convex subset $X$ of a normed vector space and with infinitely many convex constraint functions $g_i(\cdot,y): X \to \R$ indexed by indices $i$ and $y$ from a finite set $I$ and, respectively, an infinite compact metric space $Y$. 
A simple but important application of semi-infinite optimization is shape-constrained regression~\cite{GrJo, CoSaMi15, KuSc20}, that is, the problem to find a model for an unknown functional relationship of interest such that, on the one hand, the model optimally fits the available experimental data and such that, on the other hand, the model strictly respects 
the available theoretical prior knowledge about the shape of the functional relationship to be modeled. Common examples of such shape knowledge are given by monotonicity or convexity knowledge which, expressed in terms of partial derivatives, takes the form of semi-infinite constraints as in~\eqref{eq:SIP}. 
\smallskip

What we contribute 
in this paper is two adaptive discretization algorithms for convex semi-infinite optimization problems of the form~\eqref{eq:SIP}, which compute an approximate solution of~\eqref{eq:SIP} of arbitrary precision in finitely many iterations and then terminate. In other words, for every given approximation error tolerance $\delta > 0$, our algorithms terminate at a $\delta$-approximate solution of~\eqref{eq:SIP} after finitely many iterations (the total number of iterations depending on $\delta$, of course). As usual, a $\delta$-approximate solution of~\eqref{eq:SIP} 
is a point $x^* \in X$ which is feasible for~\eqref{eq:SIP} and for which $f(x^*)$ exceeds the optimal value of~\eqref{eq:SIP} at most by $\delta$, for short:
\begin{align} \label{eq:delta-approx-sol-SIP}
x^* \in F_0(Y) \qquad \text{and} \qquad f(x^*) \le \min_{x\in F_0(Y)} f(x) + \delta,
\end{align}
where $F_0(Y)$ stands for the feasibility set of~\eqref{eq:SIP}. In virtue of the feasibility of approximate solutions, our algorithms in particular -- and in sharp contrast to the algorithm from~\cite{ZhWuLo10} -- belong to the class of finitely feasible algorithms. 
In recent years, a lot of such finitely feasible algorithms have been developed for general non-convex semi-infinite optimization problems \cite{Mi11, TsRu11, DjMi17, HaPaTr19} -- including generalized, mixed-integer, and coupling-equality- and existence-constrained semi-infinite optimization problems~\cite{Mi11, MiTs15, DjGlMi19, DjMi21}. 
\smallskip

Compared to these algorithms, the two algorithms developed here require considerably smaller discretizations and are thus computationally favorable. Indeed, in updating their discretizations from iteration to iteration, our algorithms do not only add new points but also drop old points from the previous discretization.  
Also, at least our first algorithm is conceptually 
simpler than the algorithms from \cite{Mi11, DjMi17}. 
And finally, our algorithms for convex semi-infinite optimization problems terminate at approximate solutions of arbitrary precision. In contrast, for general semi-infinite problems as treated in~\cite{Mi11, TsRu11, DjMi17}, the best possible approximate-solution precision can be arbitrarily bad. 
In fact, this can happen already for just quasi-convex constraint functions. 
\smallskip

We now describe our two algorithms in a bit more detail. As in~\cite{Mi11, MiTs15, DjMi17, DjGlMi19, DjMi21} 
the restrictions
\begin{align} \label{eq:SIP_-eps}
\min_{x\in X} f(x) \quad \text{s.t.} \quad g_i(x,y) \le -\eps \quad \text{for all } y \in Y \text{ and } i \in I
\end{align}
of~\eqref{eq:SIP} with some restriction parameter $\eps \in [0,\infty)$ play an important role here. At the heart of our algorithms is an adaptive discretization algorithm in the spirit of~\cite{BlFa76}, which for a given restriction parameter $\eps$ iteratively solves discretized versions of~\eqref{eq:SIP_-eps} up to small approximation errors. We will refer to this algorithm as the core algorithm in the following. As we will see, this core algorithm for given $\eps > 0$ terminates after finitely many iterations at a point $x_{-\eps}^* \in X$ which, on the one hand, is feasible for~\eqref{eq:SIP} and for which, on the other hand, $f(x_{-\eps}^*)$ exceeds the optimal value of the restriction~\eqref{eq:SIP_-eps} only slightly. In short, 
\begin{align} \label{eq:core-algo}
x_{-\eps}^* \in F_0(Y) \qquad \text{and} \qquad f(x_{-\eps}^*) \le \min_{x\in F_{-\eps}(Y)} f(x) + \delta^*
\end{align}  
for some small $\delta^* > 0$, where $F_{-\eps}(Y)$ denotes the feasibility set of the restricted optimization problem~\eqref{eq:SIP_-eps}. We will also see 
that the optimal value of the restricted problem~\eqref{eq:SIP_-eps}, under our convexity assumptions, converges at a certain rate to the optimal value of the original problem~\eqref{eq:SIP}, that is, 
\begin{align} \label{eq:convergence-of-opt-values}
\min_{x\in F_{-\eps}(Y)} f(x) \searrow \min_{x\in F_0(Y)} f(x)
\qquad (\eps \searrow 0). 
\end{align}
In view of~\eqref{eq:core-algo} and~\eqref{eq:convergence-of-opt-values}, it is not surprising that by applying the core algorithm in a suitable way with suitably decreased restriction parameters $\eps$, 
one can construct 
algorithms that finitely terminate at approximate solutions of the original problem~\eqref{eq:SIP} of arbitrary precision. 
\smallskip

In our first algorithm -- termed the sequential algorithm --  we sequentially apply the core algorithm again and again, each time with a smaller restriction parameter $\eps > 0$. In our second algorithm -- termed the simultaneous algorithm -- we follow~\cite{Mi11, DjMi17}, 
namely, we simultaneously apply the core algorithm for $\eps >0$ and the core algorithm for $\eps = 0$. Compared to the sequential algorithm, the decreasing mechanism for $\eps$ is more involved here, because it depends on quantities calculated in the course of the algorithm. Similarly, the termination index in the simultaneous algorithm depends on such a posteriori quantities, while in the sequential algorithm the termination index  depends only on input quantities and thus is known in advance. 
\smallskip

All finite optimization subproblems occurring in our two algorithms, that is, all discretized versions of~\eqref{eq:SIP_-eps} and all feasibility test problems
\begin{align} \label{eq:Aux_i(x*)}
\max_{y\in Y} g_i(x^*,y) 
\qquad (i \in I)
\end{align} 
for given points $x^* \in X$, have to be solved only approximately -- with decreasing approximation error tolerances. Concerning the regularity of the objective and constraint functions $f: X \to \R$ and $g_i: X\times Y \to \R$, no differentiability has to be assumed -- mere continuity is sufficient. And finally, the domains $X$ and $Y$ do not have to be embedded in finite-dimensional vector spaces -- it suffices when they are a compact convex subset of any normed vector space and, respectively, a compact metric space. In particular, our algorithms can be applied in the same way both to parametric and to non-parametric shape-constrained regression problems. In the particularly important case of polynomial regression under shape constraints, the feasibility test problems~\eqref{eq:Aux_i(x*)} are global (multivariate) polynomial optimization problems and, accordingly, tailor-made solvers like~\cite{HeLa02, HaJi03} can be used to approximately solve these problems.  
\smallskip 

In the entire paper, we will adopt 
the following notation and terminology. We will often write $\mathrm{SIP}_{-\eps}(Y^*)$ for the relaxation 
\begin{align} \label{eq:SIP_-eps(Y*)}
\min_{x\in X} f(x) \quad \text{s.t.} \quad g_i(x,y) \le -\eps \quad \text{for all } y \in Y^* \text{ and } i \in I
\end{align} 
of~\eqref{eq:SIP_-eps} with $Y^*$ being an arbitrary subset of $Y$ and we will write $\mathrm{Aux}_i(x^*)$ for the auxiliary, feasibility test problems~\eqref{eq:Aux_i(x*)} for the point $x^* \in X$. As above, 
\begin{align}
F_{-\eps}(Y^*) := \{x \in X: g_i(x,y) \le -\eps \text{ for all } y \in Y \text{ and } i \in I \}
\end{align}
denotes the feasibility set of~\eqref{eq:SIP_-eps(Y*)}. A $\delta$-approximate solution of~\eqref{eq:SIP_-eps(Y*)} is defined analogously to~\eqref{eq:delta-approx-sol-SIP} and a $\delta$-approximate solution of the maximization problem~\eqref{eq:Aux_i(x*)} is defined as a $\delta$-approximate solution of the corresponding minimization problem with objective function $-g_i(x^*,\cdot)$, of course. Also, $\omega_f: [0,\infty) \to \R$ denotes the modulus of continuity of the function $f: X \to \R$, that is,
\begin{align} \label{eq:modulus-of-continuity}
\omega_f(r) := \sup\{|f(a)-f(b)|: a,b \in X \text{ with } \norm{a-b} \le r\}
\qquad (r\in[0,\infty))
\end{align}
with $\norm{\cdot}$ being the norm of the normed space in which $X$ is embedded. And finally, monotonic increasing- or decreasingness is always understood in the non-strict sense (that is, with non-strict inequalities).  


\section{Some preliminaries}

We begin with an elementary solvability result, 
which will continually -- and mostly tacitly -- be used later on. 

\begin{cond} \label{cond:cont}
$f \in C(X,\R)$ and $g_i \in C(X \times Y,\R)$ for every $i \in I$, where $X \ne \emptyset$ and $Y \ne \emptyset$ are compact metric spaces 
and $I \ne \emptyset$ is a finite index set.
\end{cond}

\begin{cond} \label{cond:strictly-convex}
$f$ and $g_i$ are as in Condition~\ref{cond:cont}. Additionally, $X$ is a convex subset of a normed vector space (the metric of $X$ being induced by the vector norm), $f$ is strictly convex, and $g_i(\cdot,y)$ is convex for every $i \in I$ and $y \in Y$.  
\end{cond}

\begin{prop} \label{prop:ex-of-SIP-solutions}
Suppose that Condition~\ref{cond:cont} is satisfied. As soon as the optimization problem~\eqref{eq:SIP_-eps(Y*)} with some arbitrary subset $Y^*$ of $Y$ and some arbitrary $\eps \in [0,\infty)$ is feasible, it has a solution. 
If even Condition~\ref{cond:strictly-convex} is satisfied, this solution is unique. 
\end{prop}

\begin{proof}
Since the feasibility set $F_{-\eps}(Y^*)$ of~\eqref{eq:SIP_-eps(Y*)} is a non-empty compact set by assumption, 
the optimization problem~\eqref{eq:SIP_-eps(Y*)} has a solution by the assumed continuity of $f$. In case even Condition~\ref{cond:strictly-convex} is satisfied, then $F_{-\eps}(Y^*)$ is also convex and therefore the problem~\eqref{eq:SIP_-eps(Y*)}, by the strict convexity of $f$, cannot have more than one solution. 
\end{proof}


We continue with a fundamental convergence result. It states that in the case of strict feasibility and of convex constraint functions, the optimal values of the restricted optimization problems~\eqref{eq:SIP_-eps} converge to the optimal value of the original optimization problem~\eqref{eq:SIP} at a rate determined by the modulus of continuity of $f$. 

\begin{prop} \label{prop:opt-values-f*_-eps-converge-to-f*}
Suppose that $f \in C(X,\R)$ and $g_i \in C(X\times Y,\R)$ for every $i \in I$, where $X \ne \emptyset$ is a convex compact subset of a normed vector space, $Y \ne \emptyset$ is a compact metric space and $I \ne \emptyset$ is a finite index set. Suppose further that~\eqref{eq:SIP} is strictly feasible, that is, 
\begin{align} \label{eq:SIP-strictly-feasible}
F_{0-}(Y) := \{x \in X: g_i(x,y) < 0 \text{ for all } y \in Y \text{ and } i \in I\} \ne \emptyset
\end{align}
and that $g_i(\cdot,y)$ is convex for every $i \in I$ and $y \in Y$. Then  
for every $\eps^* > 0$ with $F_{-\eps^*}(Y) \ne \emptyset$ the estimate
\begin{align} \label{eq:upper-bds-converge-to-f*-convergence-rate}
\min_{x \in F_{-\eps}(Y)} f(x) \le \min_{x \in F_0(Y)} f(x) + \omega_f(\diam X \cdot \eps/\eps^*) 
\qquad (\eps \in (0,\eps^*])
\end{align}
holds true, where $\omega_f$ is the modulus of continuity of $f$ as defined in~\eqref{eq:modulus-of-continuity}. In particular, 
\begin{align} \label{eq:upper-bds-converge-to-f*}
\min_{x \in F_{-\eps}(Y)} f(x) \searrow \min_{x\in F_0(Y)} f(x) \qquad (\eps \searrow 0). 
\end{align}
\end{prop}

\begin{proof}
As a first step, we observe that $F_{-\eps}(Y) \nearrow F_{0-}(Y)$ as $\eps \searrow 0$, that is, 
\begin{align} \label{eq:opt-values-f*_-eps-converge-to-f*,step-1}
F_{-\eps_1}(Y) \subset F_{-\eps_2}(Y) \qquad (\eps_1 \ge \eps_2) \qquad \text{and} \qquad
F_{0-}(Y) = \bigcup_{\eps > 0} F_{-\eps}(Y).
\end{align}
Indeed, the relation~(\ref{eq:opt-values-f*_-eps-converge-to-f*,step-1}.a) is trivial and the relation~(\ref{eq:opt-values-f*_-eps-converge-to-f*,step-1}.b) immediately follows from the continuity of the functions $g_i(x,\cdot)$ for $x \in X$ and the compactness of $Y$. 
\smallskip

As a second step, we show that for every $\eps^* > 0$ with $F_{-\eps^*}(Y) \ne \emptyset$ the asserted estimate~\eqref{eq:upper-bds-converge-to-f*-convergence-rate} holds true. 
So, let $\eps^* > 0$ be such that $F_{-\eps^*}(Y) \ne \emptyset$. Choose and fix an arbitrary element $x_{-\eps^*} \in F_{-\eps^*}(Y)$ and let  $x_0$ be an arbitrary solution of~\eqref{eq:SIP} (Proposition~\ref{prop:ex-of-SIP-solutions}), that is, 
\begin{align} \label{eq:opt-values-f*_-eps-converge-to-f*,step-2,1}
x_0 \in F_0(Y) \qquad \text{and} \qquad f(x_0) = \min_{x\in F_0(Y)}f(x).
\end{align}
Also, let $x_{-\eps} := (\eps/\eps^*) \, x_{-\eps^*} + (1-\eps/\eps^*) \, x_0$. 
Since $X$ and the functions $g_i(\cdot,y)$ are convex by assumption, we see for every $\eps \in (0,\eps^*]$ that $x_{-\eps} \in X$ and that
\begin{align}
g_i(x_{-\eps},y) \le (\eps/\eps^*) \, g_i(x_{-\eps^*},y) + (1-\eps/\eps^*) \, g_i(x_0,y) \le -\eps 
\qquad (y \in Y \text{ and } i \in I).
\end{align} 
Consequently, 
\begin{align} \label{eq:opt-values-f*_-eps-converge-to-f*,step-2,3}
x_{-\eps} \in F_{-\eps}(Y) \qquad (\eps \in (0,\eps^*]).
\end{align} 
Additionally, we trivially see from the definition of $x_{-\eps}$ that 
\begin{align} \label{eq:opt-values-f*_-eps-converge-to-f*,step-2,4}
\norm{x_{-\eps}-x_0} = (\eps/\eps^*) \, \norm{x_{-\eps^*}-x_0} \le \diam F_0(Y) \cdot \eps/\eps^* 
\qquad (\eps \in (0,\eps^*]). 
\end{align}
Combining now~\eqref{eq:opt-values-f*_-eps-converge-to-f*,step-2,1}, \eqref{eq:opt-values-f*_-eps-converge-to-f*,step-2,3} and~\eqref{eq:opt-values-f*_-eps-converge-to-f*,step-2,4}, we obtain
\begin{align}
\min_{x\in F_0(Y)} f(x) = f(x_0) 
&= f(x_{-\eps}) - (f(x_{-\eps})-f(x_0)) \notag\\
&\ge \min_{x\in F_{-\eps}(Y)} f(x) - \omega_f(\diam F_0(Y) \cdot \eps/\eps^*)
\qquad (\eps \in (0,\eps^*])
\end{align}
which immediately implies the desired estimate~\eqref{eq:upper-bds-converge-to-f*-convergence-rate}. 
\smallskip

As a third step, we conclude the asserted monotonic convergence~\eqref{eq:upper-bds-converge-to-f*}. 
Indeed, let $\eps^*$ be any positive number for which $F_{-\eps^*}(Y) \ne \emptyset$. In view of the assumed strict feasibility~\eqref{eq:SIP-strictly-feasible} and of~(\ref{eq:opt-values-f*_-eps-converge-to-f*,step-1}.b) such an $\eps^*$ does exist. We then see 
by~(\ref{eq:opt-values-f*_-eps-converge-to-f*,step-1}.a) and~\eqref{eq:upper-bds-converge-to-f*-convergence-rate} that
\begin{align} \label{eq:opt-values-f*_-eps-converge-to-f*,step-3,1}
\min_{x\in F_0(Y)}f(x) \le \min_{x\in F_{-\eps}(Y)}f(x) \le \min_{x\in F_0(Y)}f(x) + \omega_f(d \cdot \eps/\eps^*)
\qquad (\eps \in (0,\eps^*]),
\end{align} 
where $d := \diam X$. Consequently, $\min_{x\in F_{-\eps}(Y)}f(x) \longrightarrow \min_{x\in F_0(Y)}f(x)$ as $\eps \searrow 0$ by virtue of the uniform continuity of $f$. 
Additionally, by~(\ref{eq:opt-values-f*_-eps-converge-to-f*,step-1}.a) this convergence is actually monotonically decreasing. In other words, \eqref{eq:upper-bds-converge-to-f*} is satisfied, as desired. 
\end{proof}

\begin{cor}
Suppose that the assumptions of the previous proposition are satisfied. If, in addition, $f$ is even Lipschitz constinuous, then the rate of convergence in~\eqref{eq:upper-bds-converge-to-f*} is linear. 
\end{cor}

\begin{proof}
An immediate consequence of~\eqref{eq:upper-bds-converge-to-f*-convergence-rate} and the fact that
\begin{align}
\omega_f(r) \le L^* r \qquad (r \in [0,\infty))
\end{align}
for every Lipschitz constant $L^*$ of $f$. 
\end{proof}

We point out that the convexity assumption on the constraint functions $g_i(\cdot,y)$ cannot be dropped from the above results. It cannot even be weakened to quasi-convexity because for quasi-convex $g_i(\cdot,y)$, the gap
\begin{align}
\inf_{\eps\in(0,\infty)} \min_{x\in F_{-\eps}(Y)} f(x) - \min_{x\in F_0(Y)} f(x)
\end{align} 
can already be arbitrarily large. See the example below. 
%
Also, the linear convergence rate 
from the above corollary cannot be improved to some higher-order convergence rate, in general. Indeed, this easily follows with the help affine functions $f$, $g_i$.

\begin{ex}
Set $X := [-2,2]$ and $Y := [0,1]$ and let $c > 0$ be an arbitrary positive number. Also, let $f \in C(X,\R)$ be such that the minimum of $f$ is attained outside $[-1,1]$, more precisely,  
\begin{align}
\min_{x\in X}f(x) + c \le \min_{x\in [-1,1]} f(x), 
\end{align}
and let $g \in C(X\times Y, \R)$ be defined by
\begin{align}
g(x,y) := g_0(x) := \begin{cases}
x^2-1, \quad x \in [-1,1]\\
0, \quad x \in X\setminus [-1,1]
\end{cases} 
\qquad ((x,y) \in X\times Y).
\end{align}
It is then clear that Condition~\ref{cond:cont} is satisfied and that~\eqref{eq:SIP} is strictly feasible. Also, $X$ is convex and $g(\cdot,y)$ is quasi-convex for every $y \in Y$ (that is, all sublevel sets of $g(\cdot,y)$ are convex). Yet, the strict feasibility set $F_{0-}(Y)$ is much smaller than the feasibility set $F_0(Y)$, namely
\begin{align}
F_{0-}(Y) = (-1,1) \qquad \text{while} \qquad F_0(Y) = X.
\end{align} 
Consequently, we have
\begin{align}
\min_{x\in F_{-\eps}(Y)} f(x) \ge \inf_{x \in F_{0-}(Y)} f(x) = \inf_{x \in (-1,1)} f(x) \ge \min_{x\in X}f(x) + c = \min_{x\in F_0(Y)} f(x) + c
\end{align}
for every $\eps > 0$. And therefore the convergence~\eqref{eq:upper-bds-converge-to-f*} cannot hold true. $\blacktriangleleft$
\end{ex}

\section{Convergence and termination of the core algorithm}

In this section, we introduce the core algorithm and establish corresponding convergence and termination results that are at the heart of our finite-termination results for both the sequential and the simultaneous algorithm.

\subsection{Core algorithm}

As has already been mentioned in the introduction, the core algorithm (Algorithm~\ref{algo:-eps}) is inspired by the well-known and classic algorithms from~\cite{BlFa76}. In contrast to~\cite{BlFa76} however, the core algorithm requires only approximate solutions of all arising finite optimization subproblems. And, moreover, just like~\cite{Mi11} the core algorithm features discretized versions of the restrictions~\eqref{eq:SIP_-eps} of~\eqref{eq:SIP} with general restriction parameters $\eps \in [0,\infty)$, whereas $\eps = 0$ in~\cite{BlFa76}.

\begin{algo} \label{algo:-eps} 
Input: sequences $(\ol{\delta}_{k})$, $(\ul{\delta}_{k,i})$ in $[0,\infty)$ for $i \in I$, $\rho \in [0,\infty) \cup \{\infty\}$, a finite subset $Y^0$ of $Y$. With these inputs at hand, perform the following steps: 
\begin{enumerate}
\item  Set $k = 0$

\item Compute a $\ol{\delta}_k$-approximate solution $x^k$ of $\mathrm{SIP}_{-\eps}(Y^k)$

\item Compute a $\ul{\delta}_{k, i}$-approximate solution of $\mathrm{Aux}_i(x^k)$ for every $i\in I$

\item Check the sign of $g_i(x^k,y^{k,i}) + \ul{\delta}_{k,i}$ for $i\in I$
\begin{itemize}
\item If $g_i(x^k,y^{k,i}) > -\ul{\delta}_{k, i}$ for some $i \in I$, then define the new discretization by 
\begin{align} \label{eq:def-Y^k+1} 
Y^{k+1} := Y^{k+1}_1 \cup Y^{k+1}_2
\end{align}
and return to Step~2 with $k$ replaced by $k+1$. In this updating rule, the component $Y^{k+1}_1$ is chosen as 
\begin{align} \label{eq:def-Y^k+1_1} 
Y^{k+1}_1 := \{ y \in Y^k: g_i(x^k,y) \ge -\eps - \rho \text{ for some } i \in I \} 
\end{align}
and the component $Y^{k+1}_2$ is chosen as a finite subset of $Y$ that  contains a strongest violator $y^{k,i_k}$, that is, 
\begin{align} \label{eq:def-Y^k+1_2} 
\{y^{k,i_k}\} \subset Y^{k+1}_2 \subset Y \qquad \text{and} \qquad |Y^{k+1}_2| < \infty,
\end{align}
where $i_k \in I$ is an index with $g_{i_k}(x^k,y^{k,i_k}) = \max_{i\in I}g_i(x^k,y^{k,i})$. 

\item If $g_i(x^k,y^{k,i}) \le -\ul{\delta}_{k,i}$ for all $i \in I$, then terminate.
\end{itemize}
\end{enumerate}
\end{algo}

With the choice of the parameter $\rho$, one can regulate the size of the discretizations $Y^{k+1}$ generated by the algorithm and thus its  computational time. If $\rho = \infty$, 
then the discretizations get larger and larger from iteration to iteration, that is,
\begin{align} \label{eq:successive-refinement-simple-algo}
Y^k \subset Y^{k+1} 
\end{align} 
for all iteration indices $k$. If, however $\rho \in [0,\infty)$, then $Y^{k+1}_1$ will, in general, contain only some few elements of $Y^k$ and thus the 
inclusions~\eqref{eq:successive-refinement-simple-algo} will no longer be true, in general. Choosing $\rho = 0$ and $Y^{k+1}_2 = \{y^{k+1,i_k}\}$, one obtains the smallest possible new discretizations, namely
\begin{align}
Y^{k+1} = \{ y \in Y^k: g_i(x^k,y) = -\eps \text{ for some } i \in I \} \cup  \{y^{k,i_k}\}.
\end{align}
%
Algorithm~\ref{algo:-eps} essentially reduces to the algorithms from~\cite{BlFa76} in the special case 
\begin{align} \label{eq:BlFa76-spec-case}
\ol{\delta}_k, \ul{\delta}_{k,i} = 0 
\qquad \text{and} \qquad 
\eps = 0
\qquad \text{and} \qquad 
Y_2^{k+1} = \{y^{k,i_k}\}
\end{align}
where the approximate-solution tolerances $\ol{\delta}_k$ and $\ul{\delta}_{k,i}$ and the restriction parameter $\eps$ are exactly zero for all iteration indices $k$ and all $i \in I$ and where the component $Y^{k+1}_2$ solely consists of a strongest violator $y^{k,i_k}$. 
In fact, it reduces to the simple algorithm from~\cite{BlFa76} (Section~2.1) in the case~\eqref{eq:BlFa76-spec-case} with $\rho = \infty$, and to the refined algorithm from~\cite{BlFa76} (Section~2.2) in the case~\eqref{eq:BlFa76-spec-case} with $\rho = 0$.

\subsection{Convergence and termination results}

We now move on to establish central convergence and termination results for the core algorithm. Apart from the convexity and continuity assumptions made in Condition~\ref{cond:strictly-convex}, we will also need that the approximate-solution tolerances $\ol{\delta}_k$ and $\ul{\delta}_{k,i}$ converge to $0$ as $k\to\infty$, and sufficiently fast in the former case.

\begin{cond} \label{cond:params}
$\ol{\delta}_k, \ul{\delta}_{k, i} \in [0,\infty)$ and $\rho \in [0,\infty) \cup \{\infty\}$ such that $\ul{\delta}_{k, i} \longrightarrow 0$ as $k \to \infty$ for every $i \in I$ and such that 
one of the following two conditions is satisfied:
\begin{multicols}{2}
\begin{itemize}
\item[(i)] $(\ol{\delta}_k)$ is eventually $0$
\item[(ii)] $(\ol{\delta}_k)$ is summable and $\rho \ne 0$.
\end{itemize}
\end{multicols}
\end{cond}

As a first important lemma, we show that by applying the objective function $f$ to the sequences generated by the core algorithm one essentially obtains monotonically increasing sequences. In proving this, the convexity of $f$ and of the constraint functions $g_i(\cdot,y)$ is essential (except for the case $\rho = \infty$ where the increasingness is trivial by~\eqref{eq:successive-refinement-simple-algo}).

\begin{lm} \label{lm:f(x^k)-mon-increasing}
Suppose that Condition~\ref{cond:strictly-convex} and \ref{cond:params} are satisfied. Suppose further that $(x^k)$ is generated by Algorithm~\ref{algo:-eps} with some $\eps \in [0,\infty)$ and that $(x^k)$ is non-terminating. Then there exists a constant $C \in [0,\infty)$ such that
\begin{align} \label{eq:f(x^k)-mon-increasing}
f(x^k) \le f(\lambda x^k + (1-\lambda) x^{k+1}) + C \ol{\delta}_k 
\qquad (\lambda \in [0,1] \text{ and } k \in \N_0).
\end{align}
In particular, there exists a null sequence $(\nu_k)$ in $\R$ such that the sequence $(f(x^k)+\nu_k)$ is monotonically increasing. 
\end{lm}

\begin{proof}
As a first step, we show that for every $k \in \N_0$ there exists a $\lambda_k^* \in [1/2,1)$ such that
\begin{align} \label{eq:f(x^k)-increasing-step1}
\lambda x^k + (1-\lambda) x^{k+1} \in F_{-\eps}(Y^k) 
\qquad (\lambda \in [\lambda_k^*,1]). 
\end{align}
In view of the convexity of $X$ and the convexity of the functions $g_i(\cdot,y)$ for $y \in Y$ and $i \in I$, proving~\eqref{eq:f(x^k)-increasing-step1}  
boils down to showing that, for some $\lambda_k^* \in [1/2,1)$, the estimate
\begin{align} \label{eq:f(x^k)-increasing-step1,1}
\lambda g_i(x^k,y) + (1-\lambda) g_i(x^{k+1},y) \le -\eps \qquad (y \in Y^k, i \in I \text{ and } \lambda \in [\lambda_k^*,1])
\end{align}
holds true. We have to define $\lambda_k^*$ in different ways depending on the value of $\rho$. 
Write
\begin{align*}
\eta := -\eps - \rho 
\qquad \text{and} \qquad
\ol{g}(x,y) := \max_{i\in I}g_i(x,y)
\end{align*} 
for brevity. In case $\rho = \infty$ we take 
$\lambda_k^* := \lambda^* := 1/2 \in [1/2,1)$. In case $\rho \in (0,\infty)$ we take 
$\lambda_k^* := \lambda^* \in [1/2,1)$ to be so close to $1$ that
\begin{align} \label{eq:f(x^k)-increasing-step1,lambda_k^*-rho-in-(0,infty)}
(1-\lambda^*)(|\eta| + \norm{\ol{g}}_{\infty}) \le \rho.
\end{align}
And in case $\rho = 0$ we take $\lambda_k^* \in [1/2,1)$ to be so close to $1$ that 
\begin{align} \label{eq:f(x^k)-increasing-step1,lambda_k^*-rho=0}
(1-\lambda_k^*) (|\eta| + \norm{\ol{g}}_{\infty}) \le \min\{ 1/2 \, (\eta-g_i(x^k,y)): y \in Y^k\setminus Y^{k+1}_1 \text{ and } i \in I\}
\end{align}
(which is possible because the minimum on the right-hand side is strictly positive by the definition of $Y^{k+1}_1$ and the finiteness of $Y^k$ and $I$).  
With these definitions of $\lambda_k^*$, we can now establish~\eqref{eq:f(x^k)-increasing-step1,1}. 
In the case where $\rho = \infty$, we have $Y^k = Y^{k+1}_1 \subset Y^{k+1}$. Consequently, $g_i(x^k,y) \le -\eps$ and $g_i(x^{k+1},y) \le -\eps$ for every $y \in Y^k$ and $i \in I$ by the approximate solution properties of $x^k$ and $x^{k+1}$, respectively, and therefore
\begin{align} \label{eq:f(x^k)-increasing-step1,2}
\lambda g_i(x^k,y) + (1-\lambda) g_i(x^{k+1},y) \le -\eps \qquad (y \in Y^k, i \in I \text{ and } \lambda \in [0,1]). 
\end{align}
In particular, this proves~\eqref{eq:f(x^k)-increasing-step1,1} in the case $\rho = \infty$. 
In the case where $\rho \in [0,\infty)$, we prove~\eqref{eq:f(x^k)-increasing-step1,1} first for $y \in Y^k \cap Y^{k+1}_1$ and then for $y \in Y^k \setminus Y^{k+1}_1$. 
If $y \in Y^k \cap Y^{k+1}_1$, then $y \in Y^k \cap Y^{k+1}$ and therefore
\begin{align} \label{eq:f(x^k)-increasing-step1,3}
\lambda g_i(x^k,y) + (1-\lambda) g_i(x^{k+1},y) \le -\eps \qquad (i \in I \text{ and } \lambda \in [0,1]) 
\end{align}
by the same reasoning as for~\eqref{eq:f(x^k)-increasing-step1,2}. 
If $y \in Y^k \setminus Y^{k+1}_1$, then 
\begin{align} \label{eq:f(x^k)-increasing-step1,4}
\lambda g_i(x^k,y) + (1-\lambda) g_i(x^{k+1},y) 
&= \eta + (1-\lambda)(|\eta| + g_i(x^{k+1},y)) - \lambda(\eta-g_i(x^k,y)) \notag\\
&\le -\eps -\rho + (1-\lambda_k^*)(|\eta| + \norm{\ol{g}}_{\infty}) - 1/2 \, (\eta-g_i(x^k,y)) \notag \\
&\le -\eps \qquad (i \in I \text{ and } \lambda \in [\lambda_k^*,1]) 
\end{align}
by virtue of~\eqref{eq:f(x^k)-increasing-step1,lambda_k^*-rho-in-(0,infty)} or~\eqref{eq:f(x^k)-increasing-step1,lambda_k^*-rho=0}, respectively. In particular, \eqref{eq:f(x^k)-increasing-step1,3} and~\eqref{eq:f(x^k)-increasing-step1,4} together prove~\eqref{eq:f(x^k)-increasing-step1,1} in the case $\rho \in [0,\infty)$, as desired. 
\smallskip

As a second step, we show that for every $k \in \N_0$ the estimate
\begin{align} \label{eq:f(x^k)-increasing-step2}
f(x^k) \le f(\lambda x^k + (1-\lambda)x^{k+1}) + \frac{\ol{\delta}_k}{1-\lambda_k^*} 
\qquad (\lambda \in [0,1])
\end{align} 
holds true. 
In order to do so, we consider the function $\phi_k: [0,1] \to \R$ defined by $\phi_k(\lambda) := f(\lambda x^k + (1-\lambda)x^{k+1})$ for $\lambda \in [0,1]$. Since $x^k$ is a $\ol{\delta}_k$-approximate solution of $\mathrm{SIP}_{-\eps}(Y^k)$, it follows by the first step~\eqref{eq:f(x^k)-increasing-step1} that
\begin{align} \label{eq:f(x^k)-increasing-step2,1}
\phi_k(1) = f(x^k) \le \min_{x \in F_{-\eps}(Y^k)} f(x) + \ol{\delta}_k 
&\le f(\lambda x^k + (1-\lambda)x^{k+1}) + \ol{\delta}_k  \notag \\
&= \phi_k(\lambda) + \ol{\delta}_k 
\qquad (\lambda \in [\lambda_k^*,1]).
\end{align}
Since, moreover, $\phi_k$ is convex by the convexity of $f$, it further follows by the secant-slope inequality for convex functions on the reals (Remark~IV.2.11 (b) and (c) of~\cite{AmannEscher}) and by~\eqref{eq:f(x^k)-increasing-step2,1} that
\begin{align} \label{eq:f(x^k)-increasing-step2,2}
\phi_k(1)-\phi_k(\lambda) \le \frac{\phi_k(1)-\phi_k(\lambda)}{1-\lambda}
\le \frac{\phi_k(1)-\phi_k(\lambda_k^*)}{1-\lambda_k^*}
\le \frac{\ol{\delta}_k}{1-\lambda_k^*}
\qquad (\lambda \in [0,\lambda_k^*]).
\end{align} 
Combining now~\eqref{eq:f(x^k)-increasing-step2,1} and~\eqref{eq:f(x^k)-increasing-step2,2}, we obtain~\eqref{eq:f(x^k)-increasing-step2}, as desired. 
\smallskip

As a third step, we show that the estimate~\eqref{eq:f(x^k)-mon-increasing} holds true for some $C \in [0,\infty)$. 
If Condition~\ref{cond:params}(i) is satisfied, then there is a $k_0 \in \N_0$ such that $\ol{\delta}_k = 0$ for all $k \ge k_0$. So, by the second step~\eqref{eq:f(x^k)-increasing-step2} the asserted estimate~\eqref{eq:f(x^k)-mon-increasing} holds true with 
\begin{align}
C := \max\{1/(1-\lambda_k^*): k \in \{0,\dots,k_0\}\} \in [0,\infty). 
\end{align}
If Condition~\ref{cond:params}(ii) is satisfied, then $\rho \in (0,\infty)$ and thus, by our definition, $\lambda_k^* = \lambda^*$ is independent of $k \in \N_0$. So, by the second step~\eqref{eq:f(x^k)-increasing-step2} the asserted estimate~\eqref{eq:f(x^k)-mon-increasing} holds true with 
\begin{align}
C := 1/(1-\lambda^*) \in [0,\infty). 
\end{align}

As a fourth step, we finally show that there is a null sequence $(\nu_k)$ such that $(f(x^k)+\nu_k)$ is monotonically increasing. 
Indeed, from the estimate~\eqref{eq:f(x^k)-mon-increasing} (with $\lambda :=0$), one immediately verifies that the sequence $(f(x^k)+\nu_k)$ with 
\begin{align}
\nu_k := C \sum_{l=0}^{k-1} \ol{\delta}_l - C \sum_{l=0}^{\infty} \ol{\delta}_l 
\qquad (k\in\N_0)
\end{align}
is monotonically increasing and that $(\nu_k)$ is a null sequence, as desired. 
\end{proof}

As a second important lemma, we show that the accumulation points of the (non-terminating) sequences generated by the core algorithm with a certain restriction parameter $\eps$ belong to the feasibility set $F_{-\eps}(Y)$. In proving this, also the strict convexity of the objective function $f$ is crucial. It should be noticed that $F_{-\eps}(Y)$ does not have to be assumed to be non-empty -- the non-emptiness follows by the non-termination assumption. 

\begin{lm} \label{lm:acc-pts-in-F_-eps(Y)}
Suppose that Condition~\ref{cond:strictly-convex} and \ref{cond:params} are satisfied. Suppose further that $(x^k)$ is generated by Algorithm~\ref{algo:-eps} with some $\eps \in [0,\infty)$ and that $(x^k)$ is non-terminating. Then every accumulation point of $(x^k)$ belongs to $F_{-\eps}(Y)$. 
\end{lm}

\begin{proof}
We use the abbreviations $\ol{g}(x,y) := \max_{i\in I} g_i(x,y)$ for all $(x,y) \in X\times Y$ and $\ol{y}^{k} := y^{k,i_k}$ (strongest violator) throughout the proof. 
As a first step, we observe that
\begin{align} \label{eq:acc-pts-in-F_-eps(Y),1}
g_i(x^k,y) \le \ol{g}(x^k,\ol{y}^{k}) + \ul{\delta}_{k,i}
\qquad \text{and} \qquad
\ol{g}(x^{k+1},\ol{y}^{k}) \le -\eps
\end{align}
for all $y \in Y$ and all $k\in \N_0$. 
Indeed, since $y^{k,i}$ is a $\ul{\delta}_{k,i}$-approximate solution of~$\mathrm{Aux}_i(x^k)$, we see that for every $y \in Y$ and every $i \in I$ and $k \in \N_0$
\begin{align*}
g_i(x^k,y) \le g_i(x^k,y^{k,i}) + \ul{\delta}_{k,i} \le g_{i_k}(x^k,y^{k,i_k}) + \ul{\delta}_{k,i} \le \ol{g}(x^k,\ol{y}^{k}) + \ul{\delta}_{k,i},
\end{align*}
which is~(\ref{eq:acc-pts-in-F_-eps(Y),1}.a). Since moreover $\ol{y}^k \in Y^{k+1}$ by the definition of $Y^{k+1}$ 
and since $x^{k+1}$ by its approximate solution property 
belongs to $F_{-\eps}(Y^{k+1})$, we further see that~(\ref{eq:acc-pts-in-F_-eps(Y),1}.b) is satisfied, as desired.
\smallskip

As a second step, we show that for every convergent subsequence $(x^{k_l})$ of $(x^k)$ there is another subsequence $(k_{l_j})$ such that
\begin{align} \label{eq:acc-pts-in-F_-eps(Y),2}
x^{k_{l_j}+1} \longrightarrow \lim_{l\to\infty} x^{k_l} \qquad (j\to\infty).
\end{align}
So, let $(x^{k_l})$ be a convergent subsequence of $(x^k)$. 
Since $X$ is compact, the sequence $(x^{k_l+1})$ has a convergent subsequence $(x^{k_{l_j}+1})$. 
We denote the limits of these convergent subsequences by
\begin{align}
x^* := \lim_{l\to\infty} x^{k_l} \qquad \text{and} \qquad 
x^{**} := \lim_{j\to\infty} x^{k_{l_j}+1}
\end{align} 
and will show that $x^* = x^{**}$. 
It follows from Lemma~\ref{lm:f(x^k)-mon-increasing} with $\lambda := 1/2$ that
\begin{align} \label{eq:acc-pts-in-F_-eps(Y),2.1}
f(x^*) = \lim_{j\to\infty} f(x^{k_{l_j}}) \le \lim_{j\to\infty} \bigg( f\bigg(\frac{x^{k_{l_j}} + x^{k_{l_j}+1}}{2}\bigg) + C \ol{\delta}_{k_{l_j}} \bigg) = f\bigg( \frac{x^*+x^{**}}{2}\bigg).
\end{align}  
It further follows from Lemma~\ref{lm:f(x^k)-mon-increasing} that the sequence $(f(x^k)+ \nu_k)$ is monotonically increasing for some null sequence $(\nu_k)$ and therefore
\begin{align} \label{eq:acc-pts-in-F_-eps(Y),2.2}
f(x^*) = \lim_{j\to\infty} \big( f(x^{k_{l_j}}) + \nu_{k_{l_j}} \big) = \sup_{k\in\N_0} (f(x^k) + \nu_k)
&= \lim_{j\to\infty} \big( f(x^{k_{l_j}+1}) + \nu_{k_{l_j}+1} \big) \notag \\
&= f(x^{**}).
\end{align}
Assuming now that $x^* \ne x^{**}$, we obtain from~\eqref{eq:acc-pts-in-F_-eps(Y),2.1}, the strict convexity of $f$, and~\eqref{eq:acc-pts-in-F_-eps(Y),2.2} the following strict inequality:
\begin{align}
f(x^*) \le f\bigg( \frac{x^*+x^{**}}{2}\bigg) < \frac{f(x^*) + f(x^{**})}{2} = f(x^*).
\end{align}
Contradiction! So, we must have $x^* = x^{**}$, as desired.
\smallskip

As a third step, we show that every accumulation point of $(x^k)$  belongs to $F_{-\eps}(Y)$. 
So, let $x^*$ be an accumulation point of $(x^k)$. Since $Y$ is compact, there is a subsequence $(k_l)$ such that $x^{k_l} \longrightarrow x^*$ as $l \to \infty$ and such that $(\ol{y}^{k_l})$ converges to some $y^* \in Y$.  Also, by the second step, there exists yet another subsequence $(k_{l_j})$ such that
\begin{align} \label{eq:shifted-seq-same-limit,refined}
\lim_{j\to\infty} x^{k_{l_j}+1} = x^* = \lim_{l\to\infty} x^{k_l}. 
\end{align}
Combining now~(\ref{eq:acc-pts-in-F_-eps(Y),1}.a), \eqref{eq:shifted-seq-same-limit,refined}, (\ref{eq:acc-pts-in-F_-eps(Y),1}.b), we obtain the following estimate for every $y \in Y$ and every $i \in I$:
\begin{align*}
g_i(x^*,y) = \lim_{l\to\infty} g_i(x^{k_l},y) \le \lim_{l\to\infty} \ol{g}(x^{k_l},\ol{y}^{k_l}) = \ol{g}(x^*,y^*) 
= \lim_{j\to\infty} \ol{g}(x^{k_{l_j}+1},\ol{y}^{k_{l_j}}) \le -\eps.
\end{align*}
In other words, $x^*$ belongs to $F_{-\eps}(Y)$, as desired. 
\end{proof}

\begin{cor} \label{cor:algo_0-convergent}
Suppose that Condition~\ref{cond:strictly-convex} and \ref{cond:params} are satisfied and that $F_0(Y) \ne \emptyset$. Suppose further that $(x^k)$ is generated by Algorithm~\ref{algo:-eps} with $\eps := 0$ and that $(x^k)$ is non-terminating. Then $(x^k)$ converges to the unique solution of~\eqref{eq:SIP}. 
\end{cor}

\begin{proof}
We show that $(x^k)$ has only one accumulation point, namely the unique solution $x^*$ of~\eqref{eq:SIP} (Proposition~\ref{prop:ex-of-SIP-solutions}). It is then clear that $(x^k)$ must converge to $x^*$. 
So, let $x^{**}$ be any accumulation point of $(x^k)$ and let $(x^{k_l})$ be a subsequence with $x^{k_l} \longrightarrow x^{**}$ as $l\to\infty$. Since $x^{**} \in F_0(Y)$ by Lemma~\ref{lm:acc-pts-in-F_-eps(Y)} with $\eps:= 0$, we have on the one hand that
\begin{align} \label{eq:algo_0-convergent,1}
\min_{x\in F_0(Y)} f(x) \le f(x^{**}).
\end{align}
Since $x^k$ is a $\ol{\delta}_k$-approximate solution of~$\mathrm{SIP}_{-\eps}(Y^k)$ with $\eps := 0$ and $F_0(Y^k) \supset F_0(Y)$ for every $k \in \N_0$, we have on the other hand that
\begin{align} \label{eq:algo_0-convergent,2}
f(x^k) \le \min_{x\in F_{0}(Y^k)} f(x) + \ol{\delta}_k 
\le \min_{x\in F_0(Y)} f(x) + \ol{\delta}_k 
\qquad (k\in\N_0).
\end{align}
Combining now~\eqref{eq:algo_0-convergent,1} and~\eqref{eq:algo_0-convergent,2}, we conclude that
\begin{align} \label{eq:algo_0-convergent,3}
\min_{x\in F_0(Y)} f(x) \le f(x^{**}) = \lim_{l\to\infty} f(x^{k_l}) \le  \min_{x\in F_0(Y)} f(x).
\end{align}
In conjunction with the aforementioned feasibility relation $x^{**} \in F_0(Y)$, the relation~\eqref{eq:algo_0-convergent,3} proves that $x^{**}$ is a solution of~\eqref{eq:SIP}. And therefore $x^{**} = x^*$, as desired. 
\end{proof}

\begin{cor} \label{cor:algo_-eps-terminating}
Suppose that Condition~\ref{cond:strictly-convex} and \ref{cond:params} are satisfied. Suppose further that $(x^k)$ is generated by Algorithm~\ref{algo:-eps} for some $\eps \in (0,\infty)$. Then $(x^k)$ is terminating. 
\end{cor}

\begin{proof}
Assume that $(x^k)$ is non-terminating. Since $X$ and $Y$ are compact, we can then choose convergent subsequences $(x^{k_l})$ and $(\ol{y}^{k_l})$ of $(x^k)$ and $(\ol{y}^k) := (y^{k,i_k})$ (strongest violators), respectively. We denote their limits by $x^*$ and $y^*$ and write $\ol{g}(x,y) := \max_{i\in I}g_i(x,y)$ for brevity. It further follows by the non-termination assumption that 
\begin{align} \label{eq:algo_-eps-terminating,1}
\ol{g}(x^k,\ol{y}^k) \ge g_{i_k}(x^k,y^{k,i_k}) > -\ul{\delta}_{k,i} \qquad (k\in\N_0)
\end{align}
because otherwise the algorithm would terminate for some $k \in \N_0$.  And finally it follows by the non-termination assumption and Lemma~\ref{lm:acc-pts-in-F_-eps(Y)} that every accumulation point of $(x^k)$ belongs to $F_{-\eps}(Y)$. In particular, 
\begin{align} \label{eq:algo_-eps-terminating,2}
\ol{g}(x^*,y^*) \le -\eps
\end{align} 
Combining now~\eqref{eq:algo_-eps-terminating,1} and~\eqref{eq:algo_-eps-terminating,2}, we see that
\begin{align}
0 \le \lim_{l\to\infty} \ol{g}(x^{k_l},\ol{y}^{k_l}) = \ol{g}(x^*,y^*) \le -\eps < 0. 
\end{align}
Contradiction!
\end{proof}

\section{Approximate solutions with the sequential algorithm}

In this section, we apply the core algorithm (Algorithm~\ref{algo:-eps}) with smaller and smaller restriction parameters $\eps > 0$ in sequence in order to obtain our sequential finitely terminating algorithm (Algorithm~\ref{algo:sequential}). 

\subsection{Sequential algorithm} 

\begin{algo} \label{algo:feas,finite}
Input: sequences $(\ol{\delta}_{k})$, $(\ul{\delta}_{k,i})$ in $[0,\infty)$ for $i \in I$, $\rho \in [0,\infty) \cup \{\infty\}$, $r \in (1,\infty)$, a number $\eps_{0} > 0$ and a fnite subset $Y^{0}$ of $Y$. With these inputs at hand, perform the following steps: 
\begin{enumerate}
\item Set $k=0$

\item Check whether or not the restricted discretized problem $\mathrm{SIP}_{-\eps_k}(Y^k)$ is feasible
\begin{itemize}
\item If $\mathrm{SIP}_{-\eps_k}(Y^k)$ is infeasible, then set
\begin{align*}
\eps_{k+1} := \eps_k/r \qquad \text{and} \qquad
Y^{k+1} := Y^k
\end{align*}
and return to Step 2 with $k$ replaced by $k+1$

\item If $\mathrm{SIP}_{-\eps_k}(Y^k)$ is feasible, then do the following:
\begin{enumerate}
\item[2.1] Compute a $\ol{\delta}_k$-approximate solution $x^k$ of $\mathrm{SIP}_{-\eps_k}(Y^k)$ 

\item[2.2] Compute a $\ul{\delta}_k$-approximate solution $y^{k,i}$ of  $\mathrm{Aux}_i(x^k)$ for every $i \in I$

\item[2.3] Check the sign of $g_i(x^k,y^{k,i}) + \ul{\delta}_{k,i}$ for  $i\in I$
\begin{itemize}
\item If $g_i(x^k,y^{k,i}) > -\ul{\delta}_{k,i}$ for some $i\in I$, then set
\begin{gather*}
\eps_{k+1} := \eps_k, \\ 
Y^{k+1} := \{y \in Y^k: g_i(x^k,y) \ge -\eps_k - \rho \text{ for some } i \in I\} \cup \{y^{k,i_k}\}
\end{gather*}
and return to Step 2 with $k$ replaced by $k+1$

\item If $g_i(x^k,y^{k,i}) \le -\ul{\delta}_{k,i}$ for every $i \in I$, then terminate.
\end{itemize}
\end{enumerate}
\end{itemize}
\end{enumerate}
\end{algo}

Just like in the core algorithm, the points $y^{k,i_k}$ in the above algorithm denote strongest violators. 

\begin{algo} \label{algo:sequential}
Input: sequences $(\ol{\delta}_{k})$, $(\ul{\delta}_{k,i})$ in $[0,\infty)$ for $i \in I$, $\rho \in [0,\infty) \cup \{\infty\}$, $r \in (1,\infty)$, a number $\eps_{0,0} > 0$, a fnite subset $Y^{0,0}$ of $Y$, and a termination index $m^* \in \N_0 \cup \{\infty\}$. With these inputs at hand, perform the following steps: 
\begin{enumerate}
\item Set $m=0$

\item Apply Algorithm~\ref{algo:feas,finite} with $(\ol{\delta}_{m,k}) := (\ol{\delta}_{m+k})$, $(\ul{\delta}_{k,i})$, $\rho$, $r$, $\eps_{m,0}$, $Y^{m,0}$ as inputs in order to compute in finitely many steps an $x^{m, k_m} \in X$ and an $\eps_{m, k_m} \in (0,\eps_{m,0}]$ such that
\begin{align} \label{eq:seq-algo}
x^{m, k_m} \in F_0(Y) 
\qquad \text{and} \qquad
f(x^{m, k_m}) \le \min_{x \in F_{-\eps_{m, k_m}}(Y)} f(x) + \ol{\delta}_{m, k_m}
\end{align} 

\item If $m < m^*$, then set $\eps_{m+1, 0} := \eps_{m, k_m}/r$, choose a finite subset $Y^{m+1, 0}$ of $Y$, and return to Step 2 with $m$ replaced by $m+1$

\item If $m = m^*$, then terminate.
\end{enumerate}
\end{algo}

\subsection{Approximate solutions in finitely many iterations}

With the help of the termination result for the core algorithm (Corollary~\ref{cor:algo_-eps-terminating}) and the convergence rate~\eqref{eq:upper-bds-converge-to-f*-convergence-rate}, we can now establish our main finite-termination result for the sequential algorithm. It says that for every given $\delta > 0$, the sequential algorithm with termination index $m^*$ chosen as in~\eqref{eq:m*(delta,f,eps*,eps_0,0,r,delta_m)} below terminates at a $\delta$-approximate solution of~\eqref{eq:SIP} after finitely many, namely precisely $m^*$, iterations. 

\begin{lm}  \label{lm:algo-feas-fin-terminates}
Suppose that Condition~\ref{cond:strictly-convex} and~\ref{cond:params} are satisfied and that~\eqref{eq:SIP} is strictly feasible. Suppose further that $(x^k)$ and $(\eps_k)$ are generated by Algorithm~\ref{algo:feas,finite}. Then $(x^k)$ terminates at some $k^* \in \N_0$ and 
\begin{gather}
x^{k^*} \in F_0(Y) \qquad \text{and} \qquad f(x^{k^*}) \le \min_{x \in F_{-\eps_{k^*}}(Y)} f(x) + \ol{\delta}_{k^*} \label{eq:algo-feas-fin-terminates,1}\\
\eps_{k^*} \in (0,\eps_0] \label{eq:algo-feas-fin-terminates,2}
\end{gather}
\end{lm}

\begin{proof}
In the entire proof, we denote the set of iteration indices by $K$.
As a first step, we show that there exists a $k_0 \in \N_0$ such that $\mathrm{SIP}_{-\eps_k}(Y^k)$ is feasible for every iteration index $k \in K$ with $k \ge k_0$. 
Assume the contrary, then 
\begin{align} \label{eq:algo-feas-finite-terminates,step-1}
F_{-\eps_k}(Y^k) = \emptyset
\end{align} 
for infinitely many $k \in K$ and, in particular, $K = \N_0$. Consequently, by the definition of Algorithm~\ref{algo:feas,finite}, the restriction parameter is decreased infinitely often by the factor $r > 1$ and therefore
\begin{align} \label{eq:eps_k-to-zero,algo-feas-finite-terminates}
\eps_k \longrightarrow 0 \qquad (k\to\infty). 
\end{align}
Since now~\eqref{eq:SIP} is strictly feasible by assumption, there exists an $\eps^* > 0$ with $F_{-\eps^*}(Y) \ne \emptyset$ by virtue of~(\ref{eq:opt-values-f*_-eps-converge-to-f*,step-1}.b). And thus, by~\eqref{eq:eps_k-to-zero,algo-feas-finite-terminates} and (\ref{eq:opt-values-f*_-eps-converge-to-f*,step-1}.a), there exists a $k_0 \in \N_0$ such that 
\begin{align}
F_{-\eps_k}(Y^k) \supset F_{-\eps_k}(Y) \supset F_{-\eps^*}(Y) \ne \emptyset
\qquad (k \ge k_0).
\end{align}
Contradiction to our assumption~\eqref{eq:algo-feas-finite-terminates,step-1}!
\smallskip

As a second step, we show that $(x^k)$ terminates at some $k^* \in \N_0$ and that~\eqref{eq:algo-feas-fin-terminates,1} and~\eqref{eq:algo-feas-fin-terminates,2} are satisfied. 
Indeed, by the first step and the definition of Algorithm~\ref{algo:feas,finite}, the restriction parameter is not decreased anymore from iteration $k_0$ onwards, that is, 
\begin{align}
\eps_k = \dotsb = \eps_{k_0} =: \eps_* \qquad (K \ni k \ge k_0)
\end{align} 
We therefore see that the sequence $(x^{k_0+k})_{k\in K-k_0}$ can be generated by Algorithm~\ref{algo:-eps} with $\eps := \eps_* \in (0,\infty)$ and with input $(\ol{\delta}_{k_0+k})$, $(\ul{\delta}_{k_0+k,i})$, $\rho$. So, by virtue of Corollary~\ref{cor:algo_-eps-terminating}, the sequence $(x^{k_0+k})_{k\in K-k_0}$ and, by extension, $(x^k)_{k\in K}$ must be terminating. We denote the terminal iteration index of the latter sequence by $k^* := \max K < \infty$. It then follows by the definition of Algorithm~\ref{algo:feas,finite} 
that
\begin{align}
g_i(x^{k^*},y^{k^*,i}) \le -\ul{\delta}_{k^*,i} \qquad (i \in I).
\end{align}
Since $y^{k^*,i}$ is a $\ul{\delta}_{k^*,i}$-approximate solution of $\mathrm{Aux}_i(x^{k^*})$, we therefore see that 
\begin{align}
g_i(x^{k^*},y) \le g_i(x^{k^*},y^{k^*,i}) + \ul{\delta}_{k^*,i} \le 0
\qquad (y \in Y \text{ and } i \in I)
\end{align}
and thus~(\ref{eq:algo-feas-fin-terminates,1}.a) is satisfied. Since $x^{k^*}$ is a $\ol{\delta}_{k^*}$-approximate solution of $\mathrm{SIP}_{-\eps_{k^*}}(Y^{k^*})$ and $F_{-\eps_{k^*}}(Y^{k^*}) \supset F_{-\eps_{k^*}}(Y)$, we further see that~(\ref{eq:algo-feas-fin-terminates,1}.b) is satisfied. 
%
And finally, \eqref{eq:algo-feas-fin-terminates,2} is satisfied because the sequence $(\eps_k)$ is obviously monotonically decreasing. 
\end{proof}

\begin{thm} \label{thm:seq-algo}
Suppose that Condition~\ref{cond:strictly-convex} and~\ref{cond:params} are satisfied and that~\eqref{eq:SIP} is strictly feasible. Suppose further that $\delta > 0$, $\eps_{0,0} > 0$, $r \in (1,\infty)$ and $\eps^* > 0$ with $F_{-\eps^*}(Y) \ne \emptyset$ are given and that $m^* \in \N_0$ is chosen so large that
\begin{align} \label{eq:m*(delta,f,eps*,eps_0,0,r,delta_m)}
\eps_{0,0}/r^m \le \eps^*, 
\qquad
\omega_f(\diam X/\eps^* \cdot \eps_{0,0}/r^m) \le \delta/2,
\qquad
\ol{\delta}_m \le \delta/2 
\qquad (m \ge m^*).
\end{align}
If the sequence $(x^{m, k_m})$ is generated by Algorithm~\ref{algo:sequential} with $m^*$ as above, then it terminates after finitely many, namely $m^*$, iterations at a $\delta$-approximate solution of~\eqref{eq:SIP}.
\end{thm}

\begin{proof}
Suppose that $(x^{m,k_m})$ and $(\ol{\delta}_{m,k})$, $(\eps_{m,k_m})$ are generated by Algorithm~\ref{algo:sequential} with termination index $m^* \in \N_0$ as in~\eqref{eq:m*(delta,f,eps*,eps_0,0,r,delta_m)} above. 
Such an $m^*$ really exists because $r > 1$ and $\ol{\delta}_m \longrightarrow 0$ as $m\to\infty$ by assumption and because $\omega_f(s) \longrightarrow 0$ as $s \to 0$ by the uniform continuity of $f$. Also, notice that an $\eps^* >0$ with $F_{-\eps^*}(Y) \ne \emptyset$ as used in~\eqref{eq:m*(delta,f,eps*,eps_0,0,r,delta_m)} really exists by virtue of our strict feasibility assumption and by virtue of~(\ref{eq:opt-values-f*_-eps-converge-to-f*,step-1}.b). In the entire proof, we abbreviate the set of iteration indices by 
\begin{align}
M := \{m\in\N_0: m\le m^*\}.
\end{align}

As a first step, we show that for every $m \in M$, Algorithm~\ref{algo:feas,finite} applied to the inputs $(\ol{\delta}_{m,k})$, $(\ul{\delta}_{k,i})$, $\rho$, $r$, $\eps_{m,0}$ and $Y^{m,0}$ terminates after $k_m < \infty$ iterations and that 
\begin{gather}
x^{m,k_m} \in F_0(Y) \qquad \text{and} \qquad f(x^{m,k_m}) \le \min_{x\in F_{-\eps_{m,k_m}}(Y)} f(x) + \ol{\delta}_{m,k_m} \label{eq:seq-algo-step-1,1}\\
\eps_{m,k_m} \in (0,\eps_{m,0}]. \label{eq:seq-algo-step-1,2}
\end{gather}
Indeed, this easily follows from Lemma~\ref{lm:algo-feas-fin-terminates} by induction in $m \in M$. In particular, Algorithm~\ref{algo:sequential} (or, more precisely, its Step 2) is well-defined in the first place. 
\smallskip

As a second step, we show that the terminal point $x^{m^*,k_{m^*}}$ of the sequence $(x^{m,k_m})$ is a $\delta$-approximate solution of~\eqref{eq:SIP}. 
Indeed, by the definition of $(\ol{\delta}_{m,k})$, $(\eps_{m,0})$ and by~\eqref{eq:seq-algo-step-1,2} 
\begin{align} \label{eq:seq-algo-step-2,1}
\ol{\delta}_{m,k_m} = \ol{\delta}_{m+k_m} 
\qquad \text{and} \qquad
\eps_{m,k_m} \le \eps_{m,0} \le \eps_{0,0}/r^m 
\end{align}
for all $m \in M$. So, combining~(\ref{eq:seq-algo-step-1,1}.b) with~\eqref{eq:upper-bds-converge-to-f*-convergence-rate} and with~\eqref{eq:m*(delta,f,eps*,eps_0,0,r,delta_m)}, we see that
\begin{align} \label{eq:seq-algo-step-2,2}
f(x^{m^*,k_{m^*}}) 
&\le \min_{x\in F_0(Y)} f(x) + \omega_f(\diam X/\eps^* \cdot \eps_{m^*,k_{m^*}}) + \ol{\delta}_{m^*,k_{m^*}} \notag\\
&\le \min_{x\in F_0(Y)} f(x) + \delta.
\end{align}
In view of~(\ref{eq:seq-algo-step-1,1}.a) and~\eqref{eq:seq-algo-step-2,2} it is now clear that $x^{m^*,k_{m^*}}$ is a $\delta$-approximate solution of~\eqref{eq:SIP}, as desired. 
\end{proof}

If, in addition to the assumptions of the previous theorem, $f$ is even Lipschitz continuous, then $\omega_f(s) \le L^* s$ for every $s\in[0,\infty)$ 
and every Lipschitz constant $L^*$ of $f$ and thus there is a simple explicit expression for $m^*$ satisfying~\eqref{eq:m*(delta,f,eps*,eps_0,0,r,delta_m)}, 
solely in terms of 
\begin{align}
\delta, \eps^*, L^*, \diam X \qquad \text{and} \qquad (\ol{\delta}_{k}), r, \eps_{0,0}.
\end{align}
It should be noticed that, under our standing convexity and continuity assumptions (Condition~\ref{cond:strictly-convex}), the aforementioned additional Lipschitz continuity assumption is not very restrictive -- because, as is well-known, a convex continuous function $f: X \to \R$ is Lipschitz continuous on every compact subset of the interior of its domain. 

\begin{cor} \label{cor:seq-algo-converges}
Suppose that Condition~\ref{cond:strictly-convex} and~\ref{cond:params} are satisfied and that~\eqref{eq:SIP} is strictly feasible. 
If the sequence $(x^{m, k_m})$ is generated by Algorithm~\ref{algo:sequential} with $m^* = \infty$, then it converges to the unique solution of~\eqref{eq:SIP}.   
\end{cor}

\begin{proof}
As usual, we denote the unique solution of~\eqref{eq:SIP} by $x^*$. Since $m^* = \infty$, the sequence $(x^{m,k_m})$ is non-terminating. Combining~\eqref{eq:seq-algo} and~\eqref{eq:seq-algo-step-2,1} with the closedness of $F_0(Y)$ and the convergence~\eqref{eq:upper-bds-converge-to-f*}, we conclude that every accumulation point of $(x^{m,k_m})$ is a solution of~\eqref{eq:SIP} and is thus equal to $x^*$. And therefore $(x^{m,k_m})$ converges to $x^*$, as desired. 
\end{proof}

\section{Approximate solutions with the simultaneous algorithm}

In this section, we apply the core algorithm (Algorithm~\ref{algo:-eps}) with restriction parameter $\eps = 0$ and with restriction parameters $\eps > 0$ in parallel in order to obtain our simultaneous finitely terminating algorithm (Algorithm~\ref{algo:simultaneous}).

\subsection{Simultaneous algorithm} 

\begin{algo} \label{algo:simultaneous}
Input: sequences $(\ol{\delta}_k)$, $(\ul{\delta}_{k,i})$ in $[0,\infty)$ for $i \in I$, $\rho \in [0,\infty) \cup \{\infty\}$, $r \in (1,\infty)$, a number $\eps_{0} > 0$,  fnite subsets $\check{Y}^0$, $\hat{Y}^0$ of $Y$, and a termination tolerance $\delta^* > 0$. With these inputs at hand, perform the following steps: 
\begin{enumerate}
\item Set $k=0$

\item Compute a $\ol{\delta}_k$-approximate solution $\check{x}^k$ of $\mathrm{SIP}_0(\check{Y}^k)$ 

\item Compute a $\ul{\delta}_{k,i}$-approximate solution $\check{y}^{k,i}$ of $\mathrm{Aux}_i(\check{x}^k)$ for every $i\in I$

\item Check whether or not the restricted discretized problem $\mathrm{SIP}_{-\eps_k}(\hat{Y}^k)$ is feasible
\begin{itemize}
\item If $\mathrm{SIP}_{-\eps_k}(\hat{Y}^k)$ is infeasible, then set
\begin{align*}
\eps_{k+1} := \eps_k/r, \qquad \check{Y}^{k+1} := \check{Y}^k, \qquad \hat{Y}^{k+1} := \hat{Y}^k
\end{align*}
and return to Step~2 with $k$ replaced by $k+1$

\item If $\mathrm{SIP}_{-\eps_k}(\hat{Y}^k)$ is feasible, then do the following:
\begin{enumerate}
\item[4.1] Compute a $\ol{\delta}_k$-approximate solution $\hat{x}^k$ of $\mathrm{SIP}_{-\eps_k}(\hat{Y}^k)$ 

\item[4.2] Compute a $\ul{\delta}_k$-approximate solution $\hat{y}^{k,i}$ of $\mathrm{Aux}_i(\hat{x}^k)$ for every $i\in I$

\item[4.3] Check how far $f(\check{x}^k)$ and $f(\hat{x}^k)$ are apart from each other and check the sign of $g_i(\hat{x}^k, \hat{y}^{k,i}) + \ul{\delta}_{k,i}$ for $i \in I$
\begin{itemize}
\item If $f(\hat{x}^k) > f(\check{x}^k) + \delta^*$, then set
\begin{gather*}
\eps_{k+1} := \eps_k/r, \qquad \hat{Y}^{k+1} := \hat{Y}^k \\
\check{Y}^{k+1} := \{y \in \check{Y}^k: g_i(\check{x}^k,y) \ge -\rho \text{ for some } i\in I \} \cup \{\check{y}^{k,i_k}\}
\end{gather*}
and return to Step 2 with $k$ replaced by $k+1$

\item If $f(\hat{x}^k) \le f(\check{x}^k) + \delta^*$ and $g_i(\hat{x}^k,\hat{y}^{k,i}) > -\ul{\delta}_{k,i}$ for some $i\in I$, then set
\begin{gather*}
\eps_{k+1} := \eps_k, \qquad \check{Y}^{k+1} := \check{Y}^k, \\
\hat{Y}^{k+1} := \{y \in \hat{Y}^k: g_i(\hat{x}^k,y) \ge -\eps_k -\rho \text{ for some } i\in I \} \cup \{\hat{y}^{k,i_k}\}
\end{gather*}
and return to Step 2 with $k$ replaced by $k+1$

\item If $f(\hat{x}^k) \le f(\check{x}^k) + \delta^*$ and $g_i(\hat{x}^k,\hat{y}^{k,i}) \le -\ul{\delta}_{k,i}$ for every $i\in I$, then terminate.
\end{itemize} 
\end{enumerate}
\end{itemize}
\end{enumerate}
\end{algo}

Just like in the core algorithm, the points $\check{y}^{k,i_k}$, $\hat{y}^{k,i_k}$ in the above algorithm denote strongest violators.

\subsection{Approximate solutions in finitely many iterations}

With the help of the convergence and the termination result for the core algorithm (Corollary~\ref{cor:algo_0-convergent} and~\ref{cor:algo_-eps-terminating}) and the convergence~\eqref{eq:upper-bds-converge-to-f*}, we can now establish our main finite-termination result for the simultaneous  algorithm. It says that for every given $\delta > 0$, the simultaneous  algorithm with termination tolerance $\delta^* = \delta/2$ terminates at a $\delta$-approximate solution of~\eqref{eq:SIP} after finitely many iterations, provided that $\sup_{k\in\N_0} \ol{\delta}_k < \delta/2$.

\begin{thm} \label{thm:simult-algo}
Suppose that Condition~\ref{cond:strictly-convex} and Condtion~\ref{cond:params} are satisfied and that~\eqref{eq:SIP} is strictly feasible. Suppose further that $\delta > 0$, $\eps_{0,0} > 0$ and $r \in (1,\infty)$ are given and that  
\begin{align} \label{eq:simult-algo,ass}
\sup_{k\in\N_0} \ol{\delta}_k < \delta/2.
\end{align}
If the sequence $(\hat{x}^k)$ is generated by Algorithm~\ref{algo:simultaneous} with termination tolerance $\delta^* := \delta/2$, then it terminates after finitely many iterations at a $\delta$-approximate solution of~\eqref{eq:SIP}.
\end{thm}

\begin{proof}
Suppose that $(\check{Y}^k), (\check{x}^k), (\check{y}^{k,i})$ and $(\hat{Y}^k), (\hat{x}^k), (\hat{y}^{k,i})$ and $(\eps_k)$ are generated by Algorithm~\ref{algo:simultaneous} with termination tolerance $\delta^* := \delta/2$ and let $K$ be the index set of the aforementioned sequences or, in other words, the set of iteration indices. We will show (in three steps) that each of the three (non-termination) cases
\begin{gather}
F_{-\eps_k}(\hat{Y}^k) = \emptyset \label{eq:simult-algo,case-1}\\
F_{-\eps_k}(\hat{Y}^k) \ne \emptyset \quad \text{and} \quad f(\hat{x}^k) > f(\check{x}^k) + \delta/2 \label{eq:simult-algo,case-2}\\
F_{-\eps_k}(\hat{Y}^k) \ne \emptyset \quad \text{and} \quad f(\hat{x}^k) > f(\check{x}^k) + \delta/2 \quad \text{and} \quad \max_{i\in I} (g_i(\hat{x}^k,\hat{y}^{k,i}) + \ul{\delta}_{k,i}) > 0 
\label{eq:simult-algo,case-3}
\end{gather}
can occur only for finitely many iteration indices $k \in K$. And from this, in turn, we can easily conclude (in a fourth step) that the complementary (termination) case
\begin{align} \label{eq:simult-algo,case-4}
F_{-\eps_k}(\hat{Y}^k) \ne \emptyset \quad \text{and} \quad f(\hat{x}^k) > f(\check{x}^k) + \delta/2 \quad \text{and} \quad \max_{i\in I} (g_i(\hat{x}^k,\hat{y}^{k,i}) + \ul{\delta}_{k,i}) \le 0 
\end{align}
occurs for exactly one iteration index $k^* \in K$, that the algorithm terminates for this $k^*$, and that $\hat{x}^{k^*}$ is a $\delta$-approximate solution of the original optimization problem~\eqref{eq:SIP}. 
In the entire proof, we will denote the subset of those iteration indices for which case~\eqref{eq:simult-algo,case-1}, case~\eqref{eq:simult-algo,case-2}, case~\eqref{eq:simult-algo,case-3},  or case~\eqref{eq:simult-algo,case-4} occurs, respectively, by $K_1$, $K_2$, $K_3$, $K_4$. We will also write
\begin{align} \label{eq:simult-algo,delta_0}
\delta_0 := \delta/2 - \sup_{k\in\N_0} \ol{\delta}_k 
\end{align} 
which is strictly positive by assumption~\eqref{eq:simult-algo,ass}. 
Clearly, $K = K_1 \cup K_2 \cup K_3 \cup K_4$. 
Also, if $K_i$ is infinite for some $i \in \{1,2,3\}$, then in particular the algorithm does not terminate and therefore 
\begin{align} \label{eq:simult-algo,prelim}
K = \N_0 \qquad \text{and} \qquad K_4 = \emptyset.
\end{align}
As a final preliminary, we observe that the sequence $(\eps_k)$ is  monotonically decreasing by the definition of Algorithm~\ref{algo:simultaneous} and our assumption that $r > 1$. 
\smallskip

As a first step, we show that case~\eqref{eq:simult-algo,case-1} can occur only for finitely many iteration indices $k$ (in other words: $K_1$ is finite). 
Assume that, on the contrary, $K_1$ is infinite and denote by $k_l$ the $l$th element of $K_1$ for $l\in\N_0$. It then follows by the decreasingness of $(\eps_k)$ and the definition of Algorithm~\ref{algo:simultaneous} in the case~\eqref{eq:simult-algo,case-1} that $\eps_k \le \eps_{k_l+1} = \eps_{k_l}/r$ for every $k \ge k_l+1$ and every $l \in \N_0$ and therefore
\begin{align} \label{eq:simult-algo,step-1,1}
\eps_{k} \longrightarrow 0 \qquad (k\to\infty). 
\end{align}
So, by the assumed strict feasibility of~\eqref{eq:SIP} and by~(\ref{eq:opt-values-f*_-eps-converge-to-f*,step-1}.b), there exists a $k^* \in \N_0$ such that 
\begin{align}
F_{-\eps_k}(\hat{Y}^k) \supset F_{-\eps_k}(Y) \supset F_{-\eps_{k^*}}(Y) \ne \emptyset
\qquad (k\ge k^*). 
\end{align}
In particular, case~\eqref{eq:simult-algo,case-1} can occur at most for the -- finitely many -- iteration indices $k < k^*$. Contradiction to our assumption that $K_1$ is infinite!
\smallskip

As a second step, we show that case~\eqref{eq:simult-algo,case-2} can occur only for finitely many iteration indices $k$ (in other words: $K_2$ is finite).  
Assume that, on the contrary, $K_2$ is infinite, then the subset
\begin{align} \label{eq:simult-algo,K_2'}
K_2' := K_2 \cap \{k \ge k_1^*\} \qquad (k_1^* := \sup K_1 +1)
\end{align}
is infinite as well, because $k_1^* < \infty$ by the first step (notice that $k_1^* = -\infty$ in case $K_1 = \emptyset$). We denote by $k_l$ 
the $l$th element of $K_2'$ for $l \in \N_0$. 
As a first substep, we show that eventually $f(\hat{x}^{k_l})$ cannot be much larger than $\min_{x\in F_0(Y)} f(x)$, more precisely: there exists an $l_1^* \in \N_0$ such that
\begin{align} \label{eq:simult-algo,step-2.1}
f(\hat{x}^{k_l}) \le \min_{x\in F_0(Y)} f(x) + \ol{\delta}_{k_l} + \delta_0/2 
\qquad (l\ge l_1^*).
\end{align}
Indeed, by the decreasingness of $(\eps_k)$ and the definition of Algorithm~\ref{algo:simultaneous} in the case~\eqref{eq:simult-algo,case-2}, we have $\eps_{k_{l+1}} \le \eps_{k_l+1} = \eps_{k_l}/r$ for every $l \in \N_0$ and therefore
\begin{align} \label{eq:simult-algo,step-2.1,1}
\eps_{k_l} \longrightarrow 0 \qquad (l\to\infty). 
\end{align}
So, by the assumed strict feasibility of~\eqref{eq:SIP} and by Proposition~\ref{prop:opt-values-f*_-eps-converge-to-f*}, there exists an $l_1^* \in \N_0$ such that
\begin{align} \label{eq:simult-algo,step-2.1,2}
\min_{x\in F_{-\eps_{k_l}}(Y)} f(x) \le \min_{x\in F_0(Y)} f(x) + \delta_0/2 
\qquad (l\ge l_1^*).
\end{align}
Since, moreover, $\hat{x}^{k_l}$ is a $\ul{\delta}_{k_l}$-approximate solution of $\mathrm{SIP}_{-\eps_{k_l}}(\hat{Y}^{k_l})$ and $F_{-\eps_{k_l}}(\hat{Y}^{k_l}) \supset F_{-\eps_{k_l}}(Y)$ for every $l\in\N_0$, the asserted estimate~\eqref{eq:simult-algo,step-2.1} follows from~\eqref{eq:simult-algo,step-2.1,2}. 
As a second substep, we show that eventually $f(\check{x}^{k_l})$ cannot be much smaller than $\min_{x\in F_0(Y)} f(x)$, more precisely: there exists an $l_2^* \in \N_0$ such that
\begin{align} \label{eq:simult-algo,step-2.2}
f(\check{x}^{k_l}) \ge \min_{x\in F_0(Y)} f(x) - \delta_0/2 
\qquad (l\ge l_2^*). 
\end{align}
Indeed, by the definition of $k_1^*$ and by~\eqref{eq:simult-algo,prelim} we see that for every $k \in \N_0$ with $k \ge k_1^*$ either case~\eqref{eq:simult-algo,case-2} or case~\eqref{eq:simult-algo,case-3} must hold true, more precisely: case~\eqref{eq:simult-algo,case-2} holds if $k$ is equal to $k_l$ for some $l\in\N_0$ while case~\eqref{eq:simult-algo,case-3} holds if $k$ lies strictly between $k_l$ and $k_{l+1}$ for some $l\in\N_0$. So, by the definition of Algorithm~\ref{algo:simultaneous} in the case~\eqref{eq:simult-algo,case-2} and the case~\eqref{eq:simult-algo,case-3}, respectively, we see the following two facts: 
on the one hand, $\check{x}^{k_l}$ is a $\ol{\delta}_{k_l}$-approximate solution of $\mathrm{SIP}_0(\check{Y}^{k_l})$, $\check{y}^{k_l,i}$ is a $\ul{\delta}_{k_l,i}$-approximate solution of $\mathrm{Aux}_i(\check{x}^{k_l})$ and 
\begin{align}
\check{Y}^{k_l+1} = \{y \in \check{Y}^{k_l}: g_i(\check{x}^{k_l},y) \ge -\rho \text{ for some } i \in I\} \cup \{\check{y}^{k_l,i_{k_l}}\}
\end{align} 
for every $l\in\N_0$ and, on the other hand,  
\begin{align}
\check{Y}^{k_{l+1}} = \dotsb = \check{Y}^{k_l+1}
\end{align}
for every $l\in\N_0$. 
Consequently, $(\check{x}^{k_l})$ can be generated by Algorithm~\ref{algo:-eps} with $\eps := 0$ and with input $(\ol{\delta}_{k_l})$, $(\ul{\delta}_{k_l,i})$, $\rho$. And this, in turn, implies that $(\check{x}^{k_l})$ converges to the unique solution $x^*$ of~\eqref{eq:SIP} by virtue of Corollary~\ref{cor:algo_0-convergent}. In particular,
\begin{align}
f(\check{x}^{k_l}) \longrightarrow f(x^*) = \min_{x\in F_0(Y)} f(x)
\qquad (l\to\infty),
\end{align}
from which the asserted estimate~\eqref{eq:simult-algo,step-2.2} immediately follows. 
Combining now~\eqref{eq:simult-algo,step-2.1} and~\eqref{eq:simult-algo,step-2.2}, we conclude that
\begin{align}
f(\hat{x}^{k_l}) \le f(\check{x}^{k_l}) + \delta_0 + \ol{\delta}_{k_l}
\le f(\check{x}^{k_l}) + \delta/2 
\qquad (l\ge l^*),
\end{align}
where $l^* := \max\{l_1^*,l_2^*\}$. Contradiction to the fact that $f(\hat{x}^k) > f(\check{x}^k) + \delta/2$ for every $k \in K_2 \supset K_2' = \{k_l: l\in\N_0\}$!
\smallskip

As a third step, we show that case~\eqref{eq:simult-algo,case-3} can occur only for finitely many iteration indices $k$ (in other words: $K_3$ is finite).  
Assume that, on the contrary, $K_3$ is infinite, then the subset
\begin{align} \label{eq:simult-algo,K_3'}
K_3' := K_3 \cap \{k\ge k_1^*\} \cap \{k\ge k_2^*\} 
\qquad (k_i^* := \sup K_i + 1)
\end{align}
is infinite as well, because $k_1^* < \infty$ and $k_2^* < \infty$ by the first and the second step, respectively. We denote by $k_l$ the $l$th element of $K_3'$ for $l\in\N_0$. It follows by the definition of $k_1^*$ and $k_2^*$ and by~\eqref{eq:simult-algo,prelim} that for every $k\in\N_0$ with $k\ge \max\{k_1^*,k_2^*\}$, the case~\eqref{eq:simult-algo,case-3} must hold true. And therefore
\begin{align} \label{eq:simult-algo,step-3,1}
k_l = \min K_3' + l \qquad (l\in\N_0). 
\end{align}
So, by the definition of Algorithm~\ref{algo:simultaneous} in the case~\eqref{eq:simult-algo,case-3}, we see the following facts: on the one hand,
\begin{align}
\eps_{k_l} = \dotsb = \eps_{k_0} =: \eps_* 
\end{align} 
for every $l\in\N_0$ and thus, on the other hand, $\hat{x}^{k_l}$ is a $\ol{\delta}_{k_l}$-approximate solution of $\mathrm{SIP}_{-\eps_*}(\hat{Y}^{k_l})$, $\hat{y}^{k_l,i}$ is a $\ul{\delta}_{k_l,i}$-approximate solution of $\mathrm{Aux}_i(\hat{x}^{k_l})$ and 
\begin{align}
\hat{Y}^{k_{l+1}} = \hat{Y}^{k_l+1} = \{y \in \hat{Y}^{k_l}: g_i(\hat{x}^{k_l},y) \ge -\eps_*-\rho \text{ for some } i \in I\} \cup \{\hat{y}^{k_l,i_{k_l}}\}
\end{align}
for every $l\in\N_0$. Consequently, $(\hat{x}^{k_l})$ can be generated by Algorithm~\ref{algo:simultaneous} with $\eps := \eps_* \in (0,\infty)$ and with input $(\ol{\delta}_{k_l})$, $(\ul{\delta}_{k_l,i})$, $\rho$. And this, in turn, implies that $(\hat{x}^{k_l})$ terminates by virtue of Corollary~\ref{cor:algo_-eps-terminating}. Contradiction to the fact that $\{k_l: l\in\N_0\} = K_3'$ is infinite!
\smallskip

As a fourth and last step, we show that $(\hat{x}^k)$ terminates with a $\delta$-approximate solution $\hat{x}^{k^*}$ of~\eqref{eq:SIP}. 
Indeed, by the finiteness of $K_1, K_2, K_3$ established in the first three steps, the sequence $(\hat{x}^k)$ must terminate. (If it did not terminate, then the set of iteration indices would be $K = \N_0$ and the set of termination indices would be $K_4 = \emptyset$ and therefore we would have
\begin{align}
\N_0 = K = K_1 \cup K_2 \cup K_3 \cup K_4 = K_1 \cup K_2 \cup K_3. 
\end{align}
Contradiction to the finiteness of $K_1, K_2, K_3$!) 
We denote the iteration index for which $(\hat{x}^k)$ terminates by $k^*$. Consequently, for this $k^*$ the termination case~\eqref{eq:simult-algo,case-4} is satisfied. In particular, 
\begin{align} \label{eq:simult-algo,step-4,1}
f(\hat{x}^{k^*}) \le f(\check{x}^{k^*}) + \delta/2 
\qquad \text{and} \qquad
g_i(\hat{x}^{k^*}, \hat{y}^{k^*,i}) \le -\ul{\delta}_{k^*,i} \qquad (i \in I).
\end{align}
Since $\check{x}^{k^*}$ is a $\ol{\delta}_{k^*}$-approximate solution of $\mathrm{SIP}_0(\check{Y}^{k^*})$ and since $F_0(\check{Y}^{k^*}) \supset F_0(Y)$, it follows from~(\ref{eq:simult-algo,step-4,1}.a) and~\eqref{eq:simult-algo,ass} that
\begin{align} \label{eq:simult-algo,step-4,2}
f(\hat{x}^{k^*}) \le \min_{x\in F_0(\check{Y}^{k^*})} f(x) + \ol{\delta}_{k^*} + \delta/2 
&\le \min_{x\in F_0(Y)} f(x) + \ol{\delta}_{k^*} + \delta/2 \notag\\
&\le \min_{x\in F_0(Y)} f(x) + \delta.
\end{align}
Since, moreover, $\hat{y}^{k^*,i}$ is a $\ul{\delta}_{k^*,i}$-approximate solution of $\mathrm{Aux}_i(\hat{x}^{k^*})$, it further follows from~(\ref{eq:simult-algo,step-4,1}.b) that 
\begin{align} \label{eq:simult-algo,step-4,3}
g_i(\hat{x}^{k^*},y) \le \max_{y\in Y} g_i(\hat{x}^{k^*},y) \le g_i(\hat{x}^{k^*},\hat{y}^{k^*,i}) + \ul{\delta}_{k^*,i} \le 0
\qquad (y \in Y \text{ and } i \in I)
\end{align}
or, in other words, that $\hat{x}^{k^*}$ is feasible for~\eqref{eq:SIP}. Combining~\eqref{eq:simult-algo,step-4,2} and~\eqref{eq:simult-algo,step-4,3}, we finally see that $\hat{x}^{k^*}$ indeed is a $\delta$-approximate solution of~\eqref{eq:SIP}, as desired. 
\end{proof}

\section{Shape-constrained regression}

We now show how the algorithms and results developed above can be applied  to solve shape-constrained regression problems to arbitrary precision. Shape-constrained regression is about finding a model $u \mapsto v(u)$ for some true functional relationship of interest 
which, on the one hand, optimally fits given data points $(u_1,t_1), \dots, (u_N,t_N)$ and which, on the other hand, satisfies 
certain shape constraints known about the relationship to be modeled. 
In the following, we will express the deviation of the model $u \mapsto v(u)$ from the data in terms of the mean-squared error and, moreover, we will confine ourselves to shape constraints that can be expressed in terms of the partial derivatives of $u \mapsto v(u)$, like monotonicity constraints or convexity constraints, for instance. So, the shape-constrained regression problems considered here, take the following form:
\begin{equation}  \label{eq:SCRP-MSE}
\begin{gathered}
\min_{v\in V} \sum_{l=1}^N \big(v(u_l)-t_l\big)^2 + \sigma(v)  
\quad \text{s.t.} \quad 
\psi_i\big( u, (\partial^{\alpha} v(u))_{|\alpha|\le m}\big) \le 0 \\
\quad \text{for all } u \in U \text{ and } i \in I,
\end{gathered}
\end{equation}
where $U \subset \R^d$ and $V \subset C^m(U,\R)$ is a suitable set of model ansatz functions $v$ and where $\sigma(v)$ is a regularization term. Clearly, this is a semi-infinite optimization problem of the general form~\eqref{eq:SIP} considered above. Simple sufficient conditions for the continuity and convexity assumptions (Condtion~\ref{cond:cont} and~\ref{cond:strictly-convex}) to be satisfied are as follows.

\begin{prop}
Suppose that $U$ is compact in $\R^d$ and $V$ is a compact subset of 
$C^m(U,\R)$ and that $\psi_i: \R \times \R^{\ol{m}} \to \R$ is a locally Lipschitz continuous function for every $i \in I$, where $I$ is some finite index set and $\ol{m} := \binom{m+d}{d}$ is the number of multi-indices $\alpha \in \N_0^d$ with degree $|\alpha| \le m$. Suppose further that $\sigma: V \to [0,\infty)$ is continuous and that, for some given data points $(u_1,t_1), \dots, (u_N,t_N) \in U\times \R$, 
\begin{align}
f(v) := \sum_{l=1}^N \big(v(u_l)-t_l\big)^2 + \sigma(v)
\qquad \text{and} \qquad
g_i(v,u) := \psi_i\big( u, (\partial^{\alpha} v(u))_{|\alpha|\le m}\big)
\end{align} 
for $u \in U$ and $v \in V$. Then $f \in C(V,\R)$ and $g_i \in C(V\times U,\R)$ for every $i \in I$.
If, in addition, $V$ is convex, $\sigma$ is strictly convex and the functions $\psi_i(u,\cdot)$ are affine, 
then $f$ is strictly convex and $g_i(\cdot,u)$ is convex for every $u \in U$ and $i \in I$.  
\end{prop}

\begin{proof}
An elementary verification.
\end{proof}

Very roughly, shape-constrained regression problems can be categorized according to whether the chosen ansatz function set $V$ is finite-dimensionally parametrized 
or not: 
\begin{align*}
\dim \spn V < \infty \qquad \text{or} \qquad \dim \spn V = \infty,
\end{align*}
respectively. Accordingly, we will speak of parametric or, respectively, non-parametric shape-constrained regression problems.

\begin{ex}
A simple example of a finite-dimensionally parametrized set $V$ as above is given by
\begin{gather}
V := \{ \text{polynomials on $U$ of degree at most $n$ with coefficients contained in $W$}\} \notag \\
= \{v_w: w \in W\},
\end{gather}
where $U \subset \R^d$ is compact and $v_w(u) := \sum_{|\alpha| \le n} w_{\alpha} u^{\alpha}$ for $u\in U$ and where $W$ is a compact convex subset of $\R^{\ol{n}}$ with $\ol{n} := \binom{n+d}{d}$. As a strictly convex regularization term, we can choose $\sigma(v_w) := \varsigma |w|^2$ with some arbitrary $\varsigma > 0$. If the functions $\psi_i(u,\cdot)$ are affine, then the lower-level problems
\begin{align}
\max_{u \in U} g_i(u,v_w) = \max_{u \in U} \psi_i\big( u, (\partial^{\alpha} v_w(u))_{|\alpha|\le m}\big)
\end{align}
are multivariate polynomial optimization problems and thus tailor-made solvers like~\cite{HeLa02} (based on~\cite{La01}) or~\cite{HaJi03} can be used to approximately solve them. $\blacktriangleleft$ 
\end{ex}

\begin{ex}
A simple example of a not finite-dimensionally parametrized set $V$ as above is given by
\begin{align}
V := \{v \in C^{m+1}(U,\R): \norm{v}_{C^{m+1}} \le R \},
\end{align}
where $U \subset \R^d$ is compact and convex with non-empty interior and $R \in (0,\infty)$. 
Indeed, this set is a compact subset of $C^m(U,\R)$ by the Arzelà-Ascoli theorem. 
In order to see that $\spn V$ is indeed infinite-dimensional, notice that it contains all the monomials $U \ni u \mapsto u^{\alpha}$ for $\alpha \in \N_0^d$ and that these monomials make up an infinite linearly independent set. 
As a strictly convex regularization term, we can choose $\sigma(v) := \varsigma \norm{v}_{L^2(U,\R)}^2$ with some arbitrary $\varsigma > 0$.  
$\blacktriangleleft$
\end{ex}

\begin{small}

\end{small}


\begin{thebibliography}{}

\bibitem{AmannEscher} H. Amann, J. Escher: Analysis I, II, III. Birkh\"auser (2005, 2008, 2009)

\bibitem{BlFa76} J.W. Blankenship, J.E. Falk: \emph{Infinitely constrained optimization problems.} J. Optim. Th. Appl.~\textbf{19} (1976), 261-281

\bibitem{CoSaMi15} A. Cozad, N.V. Sahinidis, D. Miller: \emph{A combined first-principles and data-driven approach to model building} Comp. Chem. Eng. \textbf{73} (2015), 116-127

\bibitem{DjMi17} H. Djelassi, A. Mitsos: \emph{A hybrid discretization algorithm with guaranteed feasibility for the global solution of semi-infinite programs.} J. Glob. Optim.~\textbf{68} (2017), 227-253

\bibitem{DjGlMi19} H. Djelassi, M. Glass, A. Mitsos: \emph{Discretization-based algorithms for generalized semi-infinite and bilevel programs with coupling equality constraints.} J. Glob. Optim.~\textbf{75}, 341-392

\bibitem{DjMi21} H. Djelassi, A. Mitsos: \emph{Global solution of semi-infinite programs with existence constraints.} J. Glob. Optim.~\textbf{188} (2021), 863-881

\bibitem{GrJo} P. Groeneboom, G. Jongbloed: Nonparametric estimation under shape constraints. Cambridge University Press (2014)

\bibitem{HeLa02} D. Henrion, J.B. Lasserre: \emph{GloptiPoly: global optimization over polynomials
with Matlab and SeDuMi.} Conference proceedings of the 41st IEEE
Conference an Decision and Control (2002), 747-752

\bibitem{HaJi03} B. Hanzon, D. Jibetean: \emph{Global minimization of a multivariate polynomial using matrix methods.} J. Glob. Optim.~\textbf{27} (2003), 1-27

\bibitem{HaPaTr19} S.M. Harwood, D.J. Papageorgiou, F. Trespalacios: \emph{A note on semi-ifinite program bounding methods.} arXiv:191201763 (2019)

\bibitem{HeKo93} R. Hettich, K.O. Kortanek: \emph{Semi-infinite programming: theory, methods, and applications.} SIAM Review \textbf{35} (1993), 380-429

\bibitem{KuSc20} M. v. Kurnatowski, J. Schmid, P. Link, R. Zache, L. Morand, T. Kraft, I. Schmidt, A. Stoll: \emph{Compensating data shortages in manufacturing with monotonicity knowledge.} arXiv:2010.15955 (2020)

\bibitem{La01} J.B. Lasserre: \emph{Global optimization with polynomials and the problem of moments.} SIAM J. Optim.~\textbf{11} (2001), 796-817

\bibitem{Mi11} A. Mitsos: \emph{Global optimization of semi-infinite programs via restriction of the right-hand side.} Optimization \textbf{60} (2011), 1291-1308

\bibitem{MiTs15} A. Mitsos, A. Tsoukalas: \emph{Global optimization of generalized semi-infinite programs via restriction of the right hand side.} J. Glob. Optim.~\textbf{61} (2015), 1-17

\bibitem{Polak} E. Polak: Optimization. Algorithms and consistent approximations. Springer (1997)

\bibitem{Reemtsen} R. Reemtsen, J.-J. R\"uckmann: Semi-infinite programming. Kluwer Academic (1998)

\bibitem{Stein} O. Stein: Bi-level strategies in semi-infinite programming. Kluwer Academic (2003)

\bibitem{St12} O. Stein: \emph{How to solve a semi-infinite optimization problem.} Eur. J. Oper. Res. \textbf{223} (2012), 312--320

\bibitem{TsRu11} A. Tsoukalas, B. Rustem: \emph{A feasible point adaptation of the Blankenship and Falk algorithm for semi-infinite programming.} Optim. Lett. \textbf{5} (2011), 705-716

\bibitem{ZhWuLo10} L. Zhang, S.Y. Wu, M. Lopez: \emph{A new exchange method for convex semi-infinite programming.} SIAM J. Optim. \textbf{20} (2010), 2959-2977

\end{thebibliography}
\end{document}